\documentclass[11pt, a4paper]{article}
\usepackage{amssymb, amsmath, amsthm, amsfonts,enumerate}
\usepackage{amsmath,setspace,scalefnt}
\usepackage{fancyhdr}
\usepackage[british]{babel}

\usepackage[usenames,dvipsnames,svgnames,table]{xcolor}
\usepackage{graphicx,tikz,caption,subcaption}
\usepackage{tikz}
\usetikzlibrary{patterns}
\usetikzlibrary{shapes}
\usetikzlibrary{decorations.pathreplacing}
\usetikzlibrary{decorations.pathmorphing}
\usetikzlibrary{snakes}
\usepackage{bbm}
\usepackage{todonotes}

\usepackage{fullpage}
\usepackage{hyperref}
\usepackage[margin=2.8cm]{geometry}
\usepackage{enumitem,verbatim}

\tikzset{snake it/.style={decorate, decoration=snake}}
\usetikzlibrary{calc}

\newtheorem{theorem}{Theorem}[section]
\newtheorem{prop}[theorem]{Proposition}
\newtheorem{lemma}[theorem]{Lemma}

\theoremstyle{definition}

\newtheorem{defin}[theorem]{Definition}
\newtheorem{claim}[theorem]{Claim}

\newcounter{propcounter}

\newenvironment{poc}{\begin{proof}[Proof of claim]}{\end{proof}}

\newcommand{\eps}{\varepsilon}
\newcommand{\bN}{\mathbb{N}}

\newcommand{\cQ}{\mathcal{Q}}

\newcommand{\cP}{\mathcal{P}}
\newcommand{\cZ}{\mathcal{Z}}

\definecolor{darkblue}{rgb}{0,0,0.5}
\definecolor{armygreen}{rgb}{0.29, 0.33, 0.13}
\definecolor{darkmagenta}{rgb}{0.55, 0.0, 0.55}
\definecolor{lightseagreen}{rgb}{0.13, 0.7, 0.67}
\definecolor{darktangerine}{rgb}{1.0, 0.66, 0.07}
\definecolor{deepcarmine}{rgb}{0.66, 0.13, 0.24}
\definecolor{darkblue}{rgb}{0.0, 0.0, 0.55}
\definecolor{darkraspberry}{rgb}{0.53, 0.15, 0.34}
\definecolor{darkpastelgreen}{rgb}{0.01, 0.75, 0.24}
\definecolor{darkcyan}{rgb}{0.0, 0.55, 0.55}

\title{How to build a pillar: a proof of Thomassen's conjecture}

 \author{
	Irene Gil Fern\'andez
	\thanks{Mathematics Institute and DIMAP, University of Warwick, Coventry, CV4 7AL, UK. Email: {\tt irene.gil-fernandez@warwick.ac.uk}.}
	\and
	Hong Liu
	\thanks{Extremal Combinatorics and Probability Group (ECOPRO), Institute for Basic Science (IBS), Daejeon, South Korea. Email: {\tt hongliu@ibs.re.kr}. H.L. was supported by the Institute for Basic Science (IBS-R029-C4) and
	the UK Research and Innovation Future Leaders Fellowship MR/S016325/1.
	}
	}

\date{\today}

\begin{document}

	\maketitle

	\begin{abstract}
	 Carsten Thomassen in 1989 conjectured that if a graph has minimum degree more than the number of atoms in the universe ($\delta(G)\ge 10^{10^{10}}$), then it contains a \emph{pillar}, which is a graph that consists of two vertex-disjoint cycles of the same length, $s$ say, along with $s$ vertex-disjoint paths of the same length which connect matching vertices in order around the cycles. Despite the simplicity of the structure of pillars and various developments of powerful embedding methods for paths and cycles in the past three decades, this innocent looking conjecture has seen no progress to date. In this paper, we give a proof of this conjecture by building a pillar (algorithmically) in sublinear expanders.
	\end{abstract}

	\section{Introduction}\label{sec-intro}
Which structures can we guarantee in a graph $G$ by imposing a condition on the average degree of $G$, $d(G)$? Extremal problems under this setting have been well studied. For example, given any graph $H$ the average degree condition required to contain a copy of $H$ is the well-known Tur\'an problem. Such embedding results are often useful for problems in other areas. To just name a historical one: Erd\H{o}s~\cite{Erd38} in 1938 studied multiplicative Sidon sets via the embeddings of 4-cycles in graphs. It is well-known that unless $H$ is acyclic, the average degree condition needed to force a copy of $H$ must depend on the number of vertices in the host graph $G$. 

What about sparse host graphs $G$ with only constant average degree? It turns out that there are rich families of structures that can be embedded in such sparse graphs. For example, any graph with $d(G)\geq 2$ can easily be seen to contain a cycle. A cycle can be viewed as a subdivision of $K_3$, the complete graph on three vertices. Given a graph $H$, an \emph{$H$-subdivision} is a graph obtained from $H$ by subdividing each
of its edges into internally vertex-disjoint paths. The notion of subdivision connects graph theory and topology. In 1930, Kuratowski famously proved that a graph is planar if and only if it does not contain a subdivision of $K_{5}$ or the complete bipartite graph with three vertices in each
part~\cite{Kur30}. Generalising the above observation about cycles, Mader~\cite{Mader67} in 1967 proved that, for each $k\in\bN$, average degree at least some sufficiently large constant, depending only on $k$, is enough to guarantee a $K_k$-subdivision. In another direction, Bollob\'as~\cite{Bol77} in 1977 showed that additional conditions may be placed on the cycle length by raising the average degree condition. That is, for any $a\in \mathbb{N}$ and odd $b$, average degree at least some sufficiently large constant guarantees a cycle with length congruent to $a \mod b$. Other substructures whose existence is implied by sufficiently large constant average degree are $k$ vertex-disjoint cycles, for each $k$, shown by Corradi and Hajnal~\cite{CH63} in 1963, whose lengths may even be required to all be the same (see Egawa~\cite{Ega96}) or form an arithmetic progression (see Verstra\"ete~\cite{Ver00}).

More recently, the second author and Montgomery~\cite{LiuMont20} proved that, for each $k$, graphs with large average degree must contain a subdivision of $K_k$ such that all its edges are subdivided the same number of times. In~\cite{LiuMont20}, techniques were introduced to construct paths (and hence cycles) in extremely sparse graphs while controlling their length. These techniques imply, for example, that sufficiently large constant average degree guarantees a cycle whose length is a power of 2.

\subsection{Main result}
In this paper, we are interested in embedding pillars as subgraphs. A \emph{pillar} is a graph that consists of two disjoint cycles $C_1$ and $C_2$ of the same length, say $s$, with vertex-sets $V(C_1)=\{v_1,\dots,v_s\}$ and $V(C_2)=\{w_1,\dots,w_s\}$, and $s$ disjoint paths $Q_1,\dots,Q_s$, all of the same length, such that $Q_i$ is a $v_i,w_i$-path, for each $i\in[s]$; see Figure~\ref{fig:pillar}.

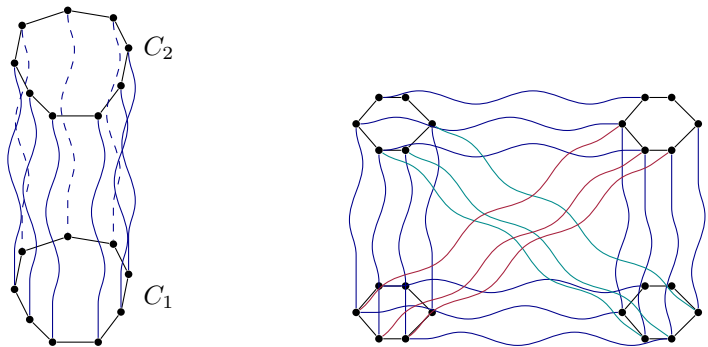
\begin{figure}[h]
\centering
\begin{tikzpicture}
	\node[inner sep= 1pt](a1) at (2.6,0.1)[circle,fill]{};
	\node[inner sep= 1pt](a2) at (2,0.1)[circle,fill]{};
	\node[inner sep= 1pt](a3) at (1.7,0.4)[circle,fill]{};
	\node[inner sep= 1pt](a4) at (1.5,0.8)[circle,fill]{};
	\node[inner sep= 1pt](a5) at (1.6,1.3)[circle,fill]{};
	\node[inner sep= 1pt](a6) at (2.2,1.5)[circle,fill]{};
	\node[inner sep= 1pt](a7) at (2.8,1.4)[circle,fill]{};
	\node[inner sep= 1pt](a8) at (3,1)[circle,fill]{};
	\node[inner sep= 1pt](a9) at (2.9,0.5)[circle,fill]{};
	\draw (a1)--(a2);
	\draw (a2)--(a3);
	\draw (a3)--(a4);
	\draw (a4)--(a5);
	\draw (a5)--(a6);
	\draw (a6)--(a7);
	\draw (a7)--(a8);
	\draw (a8)--(a9);
	\draw (a9)--(a1);
	\node[inner sep= 1pt](b1) at (2.6,3.1)[circle,fill]{};
	\node[inner sep= 1pt](b2) at (2,3.1)[circle,fill]{};
	\node[inner sep= 1pt](b3) at (1.7,3.4)[circle,fill]{};
	\node[inner sep= 1pt](b4) at (1.5,3.8)[circle,fill]{};
	\node[inner sep= 1pt](b5) at (1.6,4.3)[circle,fill]{};
	\node[inner sep= 1pt](b6) at (2.2,4.5)[circle,fill]{};
	\node[inner sep= 1pt](b7) at (2.8,4.4)[circle,fill]{};
	\node[inner sep= 1pt](b8) at (3,4)[circle,fill]{};
	\node[inner sep= 1pt](b9) at (2.9,3.5)[circle,fill]{};
	\draw (b1)--(b2);
	\draw (b2)--(b3);
	\draw (b3)--(b4);
	\draw (b4)--(b5);
	\draw (b5)--(b6);
	\draw (b6)--(b7);
	\draw (b7)--(b8);
	\draw (b8)--(b9);
	\draw (b9)--(b1);
	\draw[decorate, decoration=snake, segment length=15mm,darkblue] (b1)--(a1);
	\draw[decorate, decoration=snake, segment length=15mm,darkblue] (b2)--(a2);
	\draw[decorate, decoration=snake, segment length=15mm,darkblue] (b3)--(a3);
	\draw[decorate, decoration=snake, segment length=15mm,darkblue] (b4)--(a4);
	\draw[decorate, decoration=snake, segment length=15mm,darkblue][dashed] (b5)--(a5);
	\draw[decorate, decoration=snake, segment length=15mm,darkblue][dashed] (b6)--(a6);
	\draw[decorate, decoration=snake, segment length=15mm,darkblue][dashed] (b7)--(a7);
	\draw[decorate, decoration=snake, segment length=15mm,darkblue] (b8)--(a8);
	\draw[decorate, decoration=snake, segment length=15mm,darkblue] (b9)--(a9);
	\node[inner sep= 1pt](l1) at (3.4,0.7){\small$C_1$};
	\node[inner sep= 1pt](l2) at (3.4,4){\small$C_2$};
\end{tikzpicture}
\hspace{2cm}
\begin{tikzpicture}
	\node[inner sep= 1pt](a1) at (0.5,0.5)[circle,fill]{};
	\node[inner sep= 1pt](a2) at (0.15,0.5)[circle,fill]{};
	\node[inner sep= 1pt](a3) at (-0.15,0.85)[circle,fill]{};
	\node[inner sep= 1pt](a4) at (0.15,1.2)[circle,fill]{};
	\node[inner sep= 1pt](a5) at (0.5,1.2)[circle,fill]{};
	\node[inner sep= 1pt](a6) at (0.85,0.85)[circle,fill]{};
	\draw (a1)--(a2);
	\draw (a2)--(a3);
	\draw (a3)--(a4);
	\draw (a4)--(a5);
	\draw (a5)--(a6);
	\draw (a6)--(a1);
	\node[inner sep= 1pt](b1) at (4,0.5)[circle,fill]{};
	\node[inner sep= 1pt](b2) at (3.65,0.5)[circle,fill]{};
	\node[inner sep= 1pt](b3) at (3.35,0.85)[circle,fill]{};
	\node[inner sep= 1pt](b4) at (3.65,1.2)[circle,fill]{};
	\node[inner sep= 1pt](b5) at (4,1.2)[circle,fill]{};
	\node[inner sep= 1pt](b6) at (4.35,0.85)[circle,fill]{};
	\draw (b1)--(b2);
	\draw (b2)--(b3);
	\draw (b3)--(b4);
	\draw (b4)--(b5);
	\draw (b5)--(b6);
	\draw (b6)--(b1);
	\node[inner sep= 1pt](c1) at (4,3)[circle,fill]{};
	\node[inner sep= 1pt](c2) at (3.65,3)[circle,fill]{};
	\node[inner sep= 1pt](c3) at (3.35,3.35)[circle,fill]{};
	\node[inner sep= 1pt](c4) at (3.65,3.7)[circle,fill]{};
	\node[inner sep= 1pt](c5) at (4,3.7)[circle,fill]{};
	\node[inner sep= 1pt](c6) at (4.35,3.35)[circle,fill]{};
	\draw (c1)--(c2);
	\draw (c2)--(c3);
	\draw (c3)--(c4);
	\draw (c4)--(c5);
	\draw (c5)--(c6);
	\draw (c6)--(c1);
	\node[inner sep= 1pt](d1) at (0.5,3)[circle,fill]{};
	\node[inner sep= 1pt](d2) at (0.15,3)[circle,fill]{};
	\node[inner sep= 1pt](d3) at (-0.15,3.35)[circle,fill]{};
	\node[inner sep= 1pt](d4) at (0.15,3.7)[circle,fill]{};
	\node[inner sep= 1pt](d5) at (0.5,3.7)[circle,fill]{};
	\node[inner sep= 1pt](d6) at (0.85,3.35)[circle,fill]{};
	\draw (d1)--(d2);
	\draw (d2)--(d3);
	\draw (d3)--(d4);
	\draw (d4)--(d5);
	\draw (d5)--(d6);
	\draw (d6)--(d1);
	\draw[decorate, decoration=snake, segment length=15mm,darkblue] (b1)--(a1);
	\draw[decorate, decoration=snake, segment length=15mm,darkblue] (b3)--(a3);
	\draw[decorate, decoration=snake, segment length=15mm,darkblue] (b4)--(a4);
	\draw[decorate, decoration=snake, segment length=15mm,darkblue] (b1)--(c1);
	\draw[decorate, decoration=snake, segment length=15mm,darkblue] (b2)--(c2);
	\draw[decorate, decoration=snake, segment length=15mm,darkblue] (b3)--(c3);
	\draw[decorate, decoration=snake, segment length=15mm,darkblue] (b6)--(c6);
	\draw[decorate, decoration=snake, segment length=15mm,darkblue] (d1)--(a1);
	\draw[decorate, decoration=snake, segment length=15mm,darkblue] (d2)--(a2);
	\draw[decorate, decoration=snake, segment length=15mm,darkblue] (d3)--(a3);
	\draw[decorate, decoration=snake, segment length=15mm,darkblue] (d6)--(a6);
	\draw[decorate, decoration=snake, segment length=15mm,darkblue] (d2)--(c2);
	\draw[decorate, decoration=snake, segment length=15mm,darkblue] (d3)--(c3);
	\draw[decorate, decoration=snake, segment length=15mm,darkblue] (d4)--(c4);
	\draw[decorate, decoration=snake, segment length=15mm,darkcyan] (d1) -- (b1);
	\draw[decorate, decoration=snake, segment length=15mm,darkcyan] (d2)--(b2);
	\draw[decorate, decoration=snake, segment length=15mm,darkcyan] (d6)--(b6);
	\draw[decorate, decoration=snake, segment length=15mm,deepcarmine] (a1)--(c1);
	\draw[decorate, decoration=snake, segment length=15mm,deepcarmine] (a2)--(c2);
	\draw[decorate, decoration=snake, segment length=15mm,deepcarmine] (a3)--(c3);
\end{tikzpicture}
\caption{A pillar and a $K_4$-pillar.\label{fig:pillar}}
\end{figure}

In 1989, Thomassen~\cite{Tho89} conjectured  that every graph with sufficiently large constant minimum degree contains a pillar. There have been numerous powerful methods for embedding paths and cycles developed in the past three decades, such as Robertson and Seymour's work on graph linkage~\cite{RS95} (see also~\cite{BT96,TW05}), Bondy and Simonovits's use of Breadth First Search~\cite{BS74} (see also~\cite{Pik12}), Krivelevich and Sudakov's use of Depth First Search~\cite{KS13} and the use of expanders in a long line of work by Krivelevich (see e.g. his survey~\cite{Kri19} and more recently~\cite{EKK21}), see also a recent method developed by Gao, Huo, Liu and Ma~\cite{GHLM21}. Despite these developments and the simple nature of pillars, the innocent looking conjecture of Thomassen has seen no progress in the past thirty years. 

One explanation for the difficulty of this conjecture is the following. Cycles are not hard to embed as all vertices within are of degree 2; and subdivisions, even though having vertices of degree at least 3, can be embedded so that all these high degree vertices are pairwise far apart. Thus, embedding cycles or subdivisions boils down to anchoring at some (well-positioned) vertices and constructing vertex-disjoint paths between pairs of them. On the other hand, in a pillar, degree-3 vertices are jammed into the two cycles $C_1$ and $C_2$, which have to be embedded one next to another. To see why degree-3 vertices are game changers, a classical result of Pyber, R\"odl and Szemer\'edi~\cite{PRSz95} shows that constant average degree does \emph{not} suffice to force a 3-regular subgraph. More precisely, they constructed an $n$-vertex graph $G$ with $d(G)=\Omega(\log\log n)$ which contains no $r$-regular subgraphs for any $r\ge 3$. A priori, it is not clear whether a pillar behaves more like a subdivision or a 3-regular graph.

\smallskip

Our main result confirms Thomassen's conjecture, showing that pillars are fundamentally \emph{different} from 3-regular graphs in the sense that they can be forced by large constant degree.

\begin{theorem}\label{thm-thomassen-conj}
	There exists a constant $C>0$ such that every graph with average degree at least $C$ contains a pillar.
\end{theorem}

In fact, our method provides embeddings of more general structures. Here is one example that can be seen as a common generalisation of a subdivision and a pillar. Given $k\in\bN$, a $K_k$-\emph{pillar} is a graph that consists of $k$ disjoint cycles $C_1,\ldots, C_k$ of the same length, such that for any distinct $C_i,C_j$ there is a collection $\cQ_{i,j}$ of paths of the same length connecting matching vertices in order around $C_i, C_j$, and all paths in $\cup_{ij\in{[k]\choose 2}}\cQ_{i,j}$ are pairwise disjoint (see Figure~\ref{fig:pillar}). Note that a pillar is a $K_2$-pillar; and a $K_k$-subdivision can be obtained from taking one appropriate vertex from each cycle $C_i$ in a $K_k$-pillar along with the corresponding paths between pairs of them. Our method can be extended to show the following. We leave its proof for enthusiastic readers.

\begin{theorem}\label{thm:pillar-k}
	Given $k\in\bN$, there exists $C=C(k)>0$ such that every graph with average degree at least $C$ contains a $K_k$-pillar.
\end{theorem}

\subsection{Discussions}
\subsubsection{Our approach}
In our proof, we make crucial use of a notion called sublinear expanders (see Section~\ref{subsec-robust-expander}). We prove that pillars can be built in sublinear expanders, from which Theorem~\ref{thm-thomassen-conj} follows. There has been a sequence of advancements on the theory of sublinear expanders, which results in resolutions of several long-standing conjectures. We refer the interested readers to~\cite{GFKimKimLiu21,HKL20,CruxCycle, HHKL21, KLShS17,MadLM,LiuMont20,LWY}.

Our constructive proof can be turned into an algorithm. To deal with the troublesome degree-3 vertices in a pillar, we use a structure called \emph{kraken} (see Definition~\ref{def-kraken} and Figure~\ref{pic-kraken}). A prototype of this structure appeared in the work of Haslegrave, Kim and Liu (\emph{nakji} in~\cite{HKL20}) on sparse minors; it was formally introduced in the work of Gil Fern\'andez, Kim, Kim and Liu~\cite{GFKimKimLiu21} on cycles with geometric constraints. Roughly speaking, a kraken consists of a cycle, in which every vertex has a large `boundary'. If we manage to find two krakens, then we can link the matching vertices in their cycles by expanding and connecting their boundaries to obtain a pillar. 

Finding a \emph{single} kraken in a sublinear expander is already not an easy task; this was done in~\cite{GFKimKimLiu21} with an involved argument. To find two krakens here, we prove a \emph{robust} embedding lemma for kraken (Lemma~\ref{lem-robust-kraken}), which is the main challenge and contribution of this paper. We expect it to have further applications for embedding problems. Its proof uses the existence of kraken in sublinear expanders from~\cite{GFKimKimLiu21} as black box and builds on the techniques developed in the work of Liu and Montgomery~\cite{LiuMont20} (see the beginning of Section~\ref{sec-robust-kraken} for the high level ideas). Various difficulties occur during the embedding process as the expansion we work with is only sublinear, hence not `additive', which causes additional technicalities when implementing the above natural approach. For instance, to carry out the step of linking two krakens, we have to impose additional structural property on the krakens (see Lemma~\ref{lem-link-2pulp-adj}).

\subsubsection{Future directions}
Call a class of graphs \emph{forcible by large degree} (or forcible in short) if all graphs with large constant average degree contains one of them as a subgraph. As mentioned, by the result of Pyber, R\"odl and Szemer\'edi~\cite{PRSz95}, the class of $3$-regular graphs is not forcible, while on the contrast, our main result shows that `semi-$3$-regular' pillars are. An intriguing problem is to figure out where the line is. That is, would it be possible to give certain characterisation of close-to-3-regular graphs that are forcible? A good starting point would be to find more natural forcible classes of graphs. 

One possible concrete direction is the following. A \emph{prism} is a Cartesian product of an edge and a cycle. We can think of a pillar as a \emph{partial} subdivision of a prism, in which only the matching edges linking two cycles in the prism are subdivided. In a sense, a pillar is a minimal partial subdivision of a prism that is forcible while keeping the closeness of the degree-3 vertices. What we mean is that if we allow two consecutive matching edges in a prism to be kept unsubdivided, then the resulting class would not be forcible as such graphs all contain a 4-cycle. Indeed, it is well-known in extremal graph theory that there are $n$-vertex $4$-cycle-free graphs with average degree $\Omega(\sqrt{n})$ (consider e.g. the incidence graphs of points and lines in projective planes). In general, no upper bound can be imposed on the girth of graphs in a forcible class. 
\begin{itemize}
\item What are other obstructions to forcibility apart from bounded girth?

\item Give more (non)examples of forcible class of partial subdivision of 3-regular (or min-degree-3) graphs with adjacencies of degree-3 vertices largely preserved.
\end{itemize}

Finally, we would like to draw attention to one particular class of graphs that we do not know whether it is forcible or not. A set of $k$ edge-disjoint cycles $C_1,\ldots, C_k$ form $k$-\emph{nested cycles without crossing} if $V(C_1)\subset V(C_2)\subset \ldots \subset V(C_k)$ and for each $i\in[k-1]$, no two edges of $C_i$ are crossing chords in $C_{i+1}$ (i.e., if $C_{i+1}=v_1\dots v_{\ell}$, then $C_i$ has no two edges $v_iv_{i'}$ and $v_jv_{j'}$ with $i<j<i'<j'$). Very recently, Kim, Kim and the authors~\cite{GFKimKimLiu21} proved that $2$-nested cycles without crossing are forcible, answering an old question of Erd\H{o}s. Thomassen~\cite{Tho89} made the following stronger conjecture, which remains open: $k$-nested cycles without crossing are forcible for any fixed $k$.

\medskip

\noindent\textbf{Organisation.} Section~\ref{sec-prelim} contains the tools needed in our proofs. Theorem~\ref{thm-thomassen-conj} will be proved in Section~\ref{sec-main-proof}, which is split into two lemmas. In Section~\ref{sec-robust-kraken} we prove the key lemma that we can find a kraken robustly in a sublinear expander; and Section~\ref{sec-link-2krakens} is devoted to linking krakens using paths of the same length to obtain a pillar. 
	
	\section{Tools and building blocks}\label{sec-prelim}
	
	\subsection{Notations}
	For $n\in\mathbb{N}$, let $[n]:=\{1,\dots,n\}$. If we claim that a result holds for $0<a\ll b,c\ll d<1$, it means that there exist positive functions $f,g$ such that the result holds as long as $a<f(b,c)$ and $b<g(d)$ and $c<g(d)$. We will not compute these functions explicitly. In many cases, we treat large numbers as if they are integers, by omitting floors and ceilings if it does not affect the argument. We write $\log$ for the base-$e$ logarithm.
	
	Given a graph $G$, denote its average degree $2e(G)/|G|$ by $d(G)$, and write $\delta(G)$ and $\Delta(G)$ for its minimum and maximum degree, respectively. We write $N(v)$ for the set of neighbours of $v\in V(G)$ and we denote by $N_G(v,W)$ the set of neighbours of $v$ in $W\subseteq V(G)$ and $d_G(v,W)=|N_G(v,w)|$. Denote the (external) neighbourhood of $W$ by $N(W)=(\cup_{v\in W}N(v))\setminus W$. We write $N_G^0(W)=W$, and, for each integer $k\ge 1$, let $B_G^k(W)=\cup_{0\le j\le k}N_G^{j}(W)$ the ball of radius $k$ around $W$ in $G$, that is, the set of all vertices a graph distance at most $k$ to $W$. We let $B(W)=B^1(W)$.
	
	Let $F\subseteq G$ and $H$ be graphs, and $U\subseteq V(G)$. We write $G[U]\subseteq G$ for the induced subgraph of $G$ on vertex set $U$. Denote by $G\cup H$  the graph with vertex set $V(G)\cup V(H)$ and edge set $E(G)\cup E(H)$, and write $G-U$ for the induced subgraph $G[V(G)\setminus U]$, and $G\setminus F$ for the spanning subgraph of $G$ obtained from removing the edge set of $F$. For a path~$P$, we write $\ell(P)$ for its length, which is the number of edges in the path. Where we say $P$ is a path from a vertex set $A$ to a disjoint vertex set $B$, we mean that $P$ has one endvertex in each of $A$ and $B$, and no internal vertices in $A\cup B$.

	\subsection{3-dimensional cube in asymmetric bipartite graphs}
	The 3-dimensional cube $Q_3$ is a particular instance of the structures that we are looking for: two cycles of length $4$ whose vertices are pairwise linked by a path of length $1$. In various places when we wish to expand a set $U$ robustly, we would run into the issue that~$U$ could send most of the edges to some set $W$ that we need to avoid. In such scenarios, we can use the following simple yet useful asymmetric bipartite Tur\'an type result to obtain a 3-dimensional cube.
	
	\begin{prop}\label{prop-Q3}
		Let $d\geq 4$ be an integer and let $G$ be a bipartite graph with partite sets $U$ and $W$ such that $|U|> {|W|\choose 3}$ and every vertex in $U$ has at least $d$ neighbours in $W$. Then, $G$ contains a copy of $Q_3$.
	\end{prop}
	\begin{proof}
	    We colour triples in $W$ that have common neighbours in $U$ as follows. Consider an uncoloured triple $\{x,y,z\}$ in $W$, if they have a common neighbour $v$ in $U$ that has not yet being used to colour a triple in $N(v)$, then colour $\{x,y,z\}$ with $v$. We write $c_{x,y,z}$ for the vertex in $U$ that is used to colour $\{x,y,z\}$ if it exists. Repeat this until no more triples can be coloured. Note that in this partial colouring, no colour (in $U$) is used more than once.
		
		Since $|U|> {|W|\choose 3}$, there exists $v\in U$ that is not assigned as a colour for any triple. This, together with the maximality of the partial colouring, implies that every triple in $N(v)$ has been coloured by some vertex in $U\setminus\{v\}$. Thus, as $|N(v)|\geq d\geq 4$, we can take four vertices $x,y,z,w\in N(v)$, which together with $c_{x,y,z},c_{x,z,w},c_{y,z,w}, c_{x,y,w}$ induce a $Q_3$ in $G$.
	\end{proof}

	\subsection{Sublinear expander}\label{subsec-robust-expander}
	
	Our proof makes use of the sublinear expander introduced by Koml\'os and Szemer\'edi \cite{KSz96}. We shall use the following extension by Haslegrave, Kim and Liu~\cite{HKL20}.
	
	\begin{defin}\label{def-expander}
		Let $\varepsilon_{1}>0$ and $k\in\bN$. A graph $G$ is an $(\varepsilon_{1},k)$-\textit{expander} if for all $X\subset V(G)$ with $k/2\leq |X|\leq |G|/2$, and any subgraph $F\subseteq G$ with $e(F)\leq d(G)\cdot \eps(|X|)|X|$, we have
		$$
		|N_{G\setminus F}(X)| \geq \varepsilon(|X|)\cdot|X|,
		$$
		where
		$$
		\varepsilon(x)=\varepsilon\left(x, \varepsilon_{1}, k\right)=\left\{\begin{array}{cc}
			0 & \text { if } x<k / 5, \\
			\varepsilon_{1} / \log ^{2}(15 x / k) & \text { if } x \geq k / 5.
		\end{array}\right.
		$$ 
	\end{defin}
	Note that when $x\geq k/2$, $\eps(x)$ is decreasing, which implies that the rate of expansion, $|N_G(B^i_G(X))|/|B^i_G(X)|\geq \eps(|B^i_G(X)|,\eps_1,k)$ guaranteed by the expansion condition decreases as $i$ increases; however, $\eps(x)\cdot x$ is increasing, which means that the lower bound for $|N_G(B_G^i(X))|$ coming from the expansion property increases as $i$ increases.
	
	Such sublinear expansion rate seems rather weak at the first glance, the strength of this notion is that every graph contains one such sublinear expander subgraph with almost the same average degree. We shall use the following version, which is a combination of Lemma~3.2 in \cite{HKL20} and Corollary~2.5 in \cite{LiuMont20}.
	
	\begin{theorem}\label{thm-pass-to-expander}
		There exists some $\eps_1>0$ such that the following holds for every $\eps_2>0$ and $d\in \mathbb{N}$. Every graph $G$ with $d(G)\geq 8d$ has a bipartite $(\eps_1,\eps_2d)$-expander subgraph $H$ with $\delta(H)\geq d$.
	\end{theorem}

Thus, when dealing with extremal problems of embeddings in graphs with given density, such as proving Theorem~\ref{thm-thomassen-conj}, we can always pass to a subgraph to enjoy such expansion. One key consequence of the expansion is the so-called short diameter property. That is, we can find short paths robustly between two sufficiently large sets.

\begin{lemma}[\cite{LiuMont20}, Lemma 3.4]\label{lem-short-diam-new} 
	For each $0<\eps_1,\eps_2<1$, there exists $d_0=d_0(\eps_1,\eps_2)$ such that the following holds for each $n\geq d\geq d_0$ and $x\geq 1$. Let $G$ be an $n$-vertex $(\eps_1,\eps_2d)$-expander with $\delta(G)\geq d-1$. Let $A,B\subseteq V(G)$ with $|A|,|B|\geq x$, and  let $W\subseteq V(G)\setminus(A\cup B)$ satisfy $|W|\log^3n\leq 10x$. Then, there is a path from $A$ to $B$ in $G-W$ with length at most $\frac{40}{\eps_1}\log^3n$.
\end{lemma}

\subsection{Robust expansions of sets}
We collect here some more lemmas for robust expansion of sets in sublinear expanders.

The first one enables us to grow a set $A$ past some given set $X$ as long as $X$ does not interfere with each sphere around $A$ too much. This is formalised as follows.

\begin{defin}\label{def-thin-set}
	For $\lambda>0$ and $k\in\mathbb{N}$, we say that a vertex set $X$ in a graph $G$ is~\emph{$(\lambda,k)$-thin around $A$} if $X\cap A=\varnothing$ and, for each $i\in\mathbb{N}$,
	$$
	|N_G(B_{G-X}^{i-1}(A))\cap X|\leq \lambda \cdot i^k.
	$$
\end{defin}

We will use the following result to get such expansion.

\begin{prop}[\cite{GFKimKimLiu21}, Proposition 2.5]\label{prop-exp-HL}
	Let $0<1/d\ll \varepsilon_1\ll 1/\lambda, 1/k$ and $1\leq r\leq\log n$. Suppose $G$ is an $n$-vertex $(\varepsilon_1,\varepsilon_2d)$-expander with $\delta(G)\ge d$, and $X,Y$ are sets of vertices with $|X|\ge 1$ and $|Y|\leq \frac{1}{4}\varepsilon(|X|)\cdot |X|$. Let $W$ be a $(\lambda,k)$-thin set around $X$ in $G-Y$. Then, for each $1\leq r\leq \log n$, we have
	$$|B^r_{G-W-Y}(X)|\geq\exp(r^{1/4}).$$
\end{prop}

When we are given a large collection of sets in a sublinear expander, we can use the following lemma to find one set within the collection that expands robustly to medium (polylogarithmic) size. We remark that in the original Lemma 3.7 in \cite{LiuMont20}, condition \textbf{A3} below was stated as $C_i$ is $(4,1)$-thin around $A_i$, instead of $(\sqrt{|A_i|},1)$-thin, but the same proof works for this variant.

\begin{lemma}[\cite{LiuMont20}, Lemma~3.7]\label{lem-expand-together}
	For each $0<\eps_1<1$, $0<\eps_2<1/5$ and $k\in \bN$, there exists $d_0=d_0(\eps_1,\eps_2,k)$ such that the following holds for each $n\geq d\geq d_0$.
	Suppose that $G$ is an $n$-vertex bipartite  $(\eps_1,\eps_2 d)$-expander with $\delta(G)\ge d$.
	Let $U\subseteq V(G)$ satisfy $|U|\leq \exp((\log\log n)^2)$. Let $r\geq n^{1/8}$ and $\ell_0=(\log\log n)^{20}$. Suppose $(A_i,B_i,C_i)$, $i\in [r]$, are such that the following hold for each $i\in [r]$.
	\stepcounter{propcounter}
	\begin{enumerate}[label = \emph{\bfseries \Alph{propcounter}\arabic{enumi}}]
		\item $|A_i|\geq d_0$.\label{exptog}
		\item $B_i\cup C_i$ and $A_i$ are disjoint sets in $V(G)\setminus U$, with $|B_i|\leq |A_i|/\log^{10}|A_i|$.\label{exptog2}
		\item $C_i$ is $(\sqrt{|A_i|},1)$-thin around $A_i$ in $G-U-B_i$.\label{exptog3}
		\item Each vertex in $B_{G-U-B_i-C_i}^{\ell_0}(A_i)$ has at most $d/2$ neighbours in $U$.\label{exptog4}
		\item For each $j\in [r]\setminus\{i\}$, $A_i$ and $A_j$ are at least a distance $2\ell_0$ apart in $G-U-B_i-C_i-B_j-C_j$.\label{exptog5}
	\end{enumerate}
	Then, for some $i\in [r]$, $$|B^{\ell_0}_{G-U-B_i-C_i}(A_i)|\geq \log^{k}n.$$
\end{lemma} 

Lastly, we need the following result to find a linear size vertex set with polylogarithmic diameter in $G$ while avoiding an arbitrary set of size $o(n/\log^2n)$. 

\begin{lemma}[\cite{LiuMont20}, Lemma~3.12]\label{lem-find-large-ball}
	Let $0<1/d\ll\eps_1,\eps_2<1$ and let $G$ be an $n$-vertex bipartite $(\eps_1,\eps_2d)$-expander with $\delta(G)\geq d$. For any $W\subseteq V(G)$ with $|W|\leq \eps_1n/100\log^2n$, there is a set $B\subseteq G-W$ with size at least $n/25$ and diameter at most $200\eps_1^{-1}\log^3n$.
\end{lemma}

\subsection{Krakens}

A basic structure we often use is a large set with small radius, defined as follows.
	
	\begin{defin}\label{def-expansion}
		Given a vertex $v$ in a graph $F$, $F$ is a \emph{$(D,m)$-expansion of $v$} if $|F|=D$ and every vertex of $F$ is a distance at most $m$ from $v$.
	\end{defin}
	
	Expansions around a vertex can be trimmed to a smaller size.
	\begin{prop}[\cite{LiuMont20}, Proposition 3.10]\label{prop:trimming} Let $D,m\in \bN$ and $1\leq D'\leq D$. Then, any graph $F$ which is a $(D,m)$-expansion of $v$ contains a subgraph which is a $(D',m)$-expansion of $v$.
\end{prop}

\begin{defin}\label{def-kraken}
	For $k,s,t\in\mathbb{N}$, a \emph{$(k,s,t)$-kraken} is a graph that consists of a cycle $C$ with vertices $v_1,\dots,v_k$, vertices $u_1,\ldots, u_k$ out of $V(C)$, and subgraphs $F_{j}$ and $P_{j}$, $j\in[k]$, such that
	\begin{itemize}\itemsep=0pt
		\item $\{F_{j}:j\in[k]\}$ is a collection of sets disjoint from each other and from $V(C)$, and each~$F_j$ is a $(t,s)$-expansion of $u_j$. We call each $F_j$ a \emph{leg} and $u_j$ its \emph{end}.
		
		\item $\{P_{j}:j\in[k]\}$ is a collection of pairwise disjoint paths, and each $P_j$ is a $v_j,u_j$-path of length at most $10s$ with internal vertices disjoint from $V(C)\cup(\cup_{i\in[k]} V(F_i))$. 
	\end{itemize}
\end{defin}


\begin{figure}[h]
	\centering
	\scalebox{0.8}{	
		\begin{tikzpicture}	
			\draw[black,scale=0.7] (0,0) circle (1cm);
			\node[inner sep= 1pt](t1) at (0,0){\small$C$};
			\node[inner sep= 1pt](a1) at (0.5,0.5)[circle,fill]{};
			\node[inner sep= 1pt](a7) at (0.7,0.15)[circle,fill]{};
			\node[inner sep= 1pt](a2) at (0.2,0.68)[circle,fill]{};
			\node[inner sep= 1pt](t2) at (0.2,0.4){\small$v_j$};
			\node[inner sep= 1pt](t2) at (1.2,3.1){\small$u_j$};
			\node[inner sep= 1pt](t3) at (1,2){\small$P_j$};
			\node[inner sep= 1pt](t4) at (1.9,3){\small$F_j$};
			\node[inner sep= 1pt](a8) at (-0.2,0.68)[circle,fill]{};
			\node[inner sep= 1pt](a3) at (-0.5,0.5)[circle,fill]{};
			\node[inner sep= 1pt](a9) at (-0.7,0.1)[circle,fill]{};
			\node[inner sep= 1pt](a4) at (-0.65,-0.3)[circle,fill]{};
			\node[inner sep= 1pt](a10) at (-0.45,-0.55)[circle,fill]{};
			\node[inner sep= 1pt](a5) at (-0.1,-0.7)[circle,fill]{};
			\node[inner sep= 1pt](a11) at (0.3,-0.65)[circle,fill]{};
			\node[inner sep= 1pt](a6) at (0.68,-0.2)[circle,fill]{};
			\node[inner sep= 1pt](a12) at (0.52,-0.47)[circle,fill]{};
			\draw[lightseagreen] (3.5,0.5) circle (0.5cm);
			\draw[lightseagreen] (3,2) circle (0.5cm);
			\draw[decorate, decoration=snake, segment length=5mm,darkblue] (a7) -- (3,0.4);
			\draw[decorate, decoration=snake, segment length=6mm,darkblue] (a1) -- (2.55,1.8);
			\draw[lightseagreen] (1.2,3.4) circle (0.5cm);
			\draw[lightseagreen] (-1,3.5) circle (0.5cm);
			\draw[decorate, decoration=snake, segment length=5mm,darkblue] (a2) -- (1,2.95);
			\node[inner sep= 1pt](a15) at (1,2.95)[circle,fill]{};
			\draw[decorate, decoration=snake, segment length=5mm,darkblue] (a8) -- (-0.8,3.05);
			\draw[lightseagreen] (-3.5,0.6) circle (0.5cm);
			\draw[lightseagreen] (-3.1,2) circle (0.5cm);
			\draw[decorate, decoration=snake, segment length=5mm,darkblue] (a9) -- (-3,0.45);
			\draw[decorate, decoration=snake, segment length=4mm,darkblue] (a3) -- (-2.7,1.7);
			\draw[lightseagreen] (-3.5,-0.6) circle (0.5cm);
			\draw[lightseagreen] (-3.1,-2) circle (0.5cm);
			\draw[decorate, decoration=snake, segment length=5mm,darkblue] (a4) -- (-3,-0.45);
			\draw[decorate, decoration=snake, segment length=4mm,darkblue] (a10) -- (-2.7,-1.7);
			\draw[lightseagreen] (1.2,-3.4) circle (0.5cm);
			\draw[lightseagreen] (-1.3,-3.5) circle (0.5cm);
			\draw[decorate, decoration=snake, segment length=5mm,darkblue] (a11) -- (1,-2.95);
			\draw[decorate, decoration=snake, segment length=5mm,darkblue] (a5) -- (-1.1,-3.05);
			\draw[lightseagreen] (3.5,-1) circle (0.5cm);
			\draw[lightseagreen] (3.1,-2.5) circle (0.5cm);
			\draw[decorate, decoration=snake, segment length=5mm,darkblue] (a6) -- (3.03,-0.85);
			\draw[decorate, decoration=snake, segment length=4.5mm,darkblue] (a12) -- (2.7,-2.2);
		\end{tikzpicture}
	}
	\caption{A $(12,s,t)$-kraken.\label{pic-kraken}}
\end{figure}
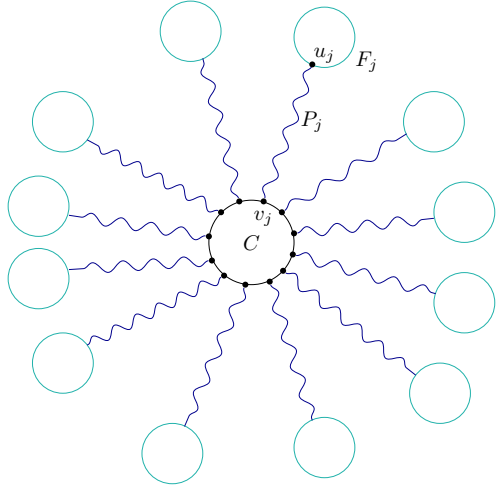

We usually write a kraken as a tuple $(C, u_j, F_{j}, P_{j})$, $j\in[k]$, see Figure~\ref{pic-kraken}.

The following result guarantees a large kraken in a sublinear expander.

\begin{lemma}[\cite{GFKimKimLiu21}, Lemma~3.2]\label{lem-kraken-from-kraken}
	Let $0<1/d\ll \eps_1,\eps_2,1/b<1$. Let $G$ be an $n$-vertex $(\eps_1,\eps_2d)$-expander with $\delta(G)\geq d$. 
	Let $m=200\eps_1^{-1}\log^3n$. Then, there exists a $(k,m,\log^{b}n)$-kraken $(C, u_j, F_j, P_j)$, $j\in[k]$, in $G$ for some $k\le \log n$.
		
\end{lemma}

\subsection{Adjusters}
An important tool we need in our proof is a recent lemma of Liu and Montgomery~\cite{LiuMont20} (Lemma~\ref{lem-finalconnect}), which robustly finds paths of specific lengths between a given pair of vertices in a sublinear expander with some mild conditions. 

Before stating the lemma, let us briefly introduce the key object called \emph{adjuster}, involved for its proof. The basic structure is an even cycle~$C$, together with two disjoint large connected subgraphs $F_1,F_2$ attached to two almost-antipodal vertices $v_1,v_2$ on the cycle $C$. If~$C$ has length $2\ell$, for some $\ell\leq\log n$, there are two $v_1,v_2$-paths, one of length $\ell+1$ and the other with length $\ell-1$. The subgraphs $F_1$, $F_2$ are set to be comfortably larger than the size of $C$, so that they can be connected by a short path while avoiding $C$. The idea is to link many such structures sequentially to form an \emph{adjuster}, so that we can use it to find paths of many different lengths by varying the length of the path we take around each cycle, see Figure~\ref{pic-adjuster}.

	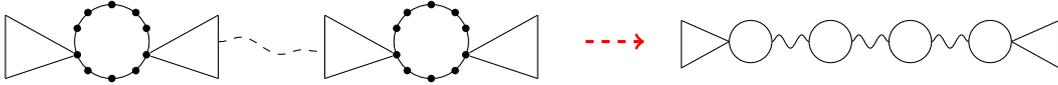
\begin{figure}[h]
		\centering
	\begin{tikzpicture}[scale=0.7]
	\draw[black] (0,0) circle (0.7cm);
	\node[inner sep= 1pt](a1) at (0.45,0.55)[circle,fill]{};
	\node[inner sep= 1pt](a2) at (0,0.7)[circle,fill]{};
	\node[inner sep= 1pt](a3) at (-0.45,0.55)[circle,fill]{};
	\node[inner sep= 1pt](a4) at (-0.45,-0.55)[circle,fill]{};
	\node[inner sep= 1pt](a5) at (0,-0.7)[circle,fill]{};
	\node[inner sep= 1pt](a6) at (0.45,-0.55)[circle,fill]{};
	\node[inner sep= 1pt](a7) at (0.65,0.25)[circle,fill]{};
	\node[inner sep= 1pt](a8) at (-0.65,0.25)[circle,fill]{};
	\node[inner sep= 1pt](a12) at (-0.65,-0.25)[circle,fill]{};
	\node[inner sep= 1pt](a13) at (0.65,-0.25)[circle,fill]{};
	\draw (a12)--(-2,0.5);
	\draw (a12)--(-2,-0.7);
	\draw (-2,0.5)--(-2,-0.7);
	\draw (a13)--(2,0.5);
	\draw (a13)--(2,-0.7);
	\draw (2,0.5)--(2,-0.7);
	\draw[black] (6,0) circle (0.7cm);
	\node[inner sep= 1pt](b1) at (6.45,0.55)[circle,fill]{};
	\node[inner sep= 1pt](b2) at (6,0.7)[circle,fill]{};
	\node[inner sep= 1pt](b3) at (5.55,0.55)[circle,fill]{};
	\node[inner sep= 1pt](b4) at (5.55,-0.55)[circle,fill]{};
	\node[inner sep= 1pt](b5) at (6,-0.7)[circle,fill]{};
	\node[inner sep= 1pt](b6) at (6.45,-0.55)[circle,fill]{};
	\node[inner sep= 1pt](b7) at (6.65,0.25)[circle,fill]{};
	\node[inner sep= 1pt](b8) at (5.35,0.25)[circle,fill]{};
	\node[inner sep= 1pt](b12) at (5.35,-0.25)[circle,fill]{};
	\node[inner sep= 1pt](b13) at (6.65,-0.25)[circle,fill]{};
	\draw (b12)--(4,0.5);
	\draw (b12)--(4,-0.7);
	\draw (4,0.5)--(4,-0.7);
	\draw (b13)--(8,0.5);
	\draw (b13)--(8,-0.7);
	\draw (8,0.5)--(8,-0.7);
	\draw[decorate, decoration=snake, segment length=10mm,black][dashed] (2,0)--(4,-0.2);
	\draw[red,very thick,dashed,->] (8.9,0) -- (10,0);
       \draw[black] (12,0) circle (0.4cm);
	   \draw[black] (13.5,0) circle (0.4cm);
	   \draw[black] (15,0) circle (0.4cm);
	   \draw[black] (16.5,0) circle (0.4cm);
	   \draw[decorate, decoration=snake, segment length=3mm,black] (12.4,0)--(13.1,0);
	   \draw[decorate, decoration=snake, segment length=3mm,black] (13.9,0)--(14.6,0);
	   \draw[decorate, decoration=snake, segment length=3mm,black] (15.4,0)--(16.1,0);
	   \draw (11.6,0)--(10.7,0.4);
	   \draw (11.6,0)--(10.7,-0.5);
	   \draw (10.7,0.4)--(10.7,-0.5);
	   \draw (16.9,0)--(17.8,0.4);
	   \draw (16.9,0)--(17.8,-0.5);
	   \draw (17.8,0.4)--(17.8,-0.5);
	\end{tikzpicture}
\caption{Adjuster. \label{pic-adjuster}}
	\end{figure}

We need the following definition to record the parity of paths between two vertices in a connected bipartite graph. For any connected bipartite graph $H$ and $u,v\in V(H)$, let
	$$
	\pi(u,v,H)=\left\{\begin{array}{ll}
		0 & \text{ if $u$ and $v$ are in the same vertex class in the (unique) bipartition of $H$},\\
		1 & \text{ if $u$ and $v$ are in different vertex classes in the bipartition of $H$}.
	\end{array}
	\right.
	$$

Using adjusters, Liu and Montgomery~\cite{LiuMont20} proved the following.

\begin{lemma}[\cite{LiuMont20}, Lemma 4.8]\label{lem-finalconnect}
	There exists some $\eps_1>0$ such that, for any $0<\eps_2<1/5$ and $b\geq 10$, there exists $d_0=d_0(\eps_1,\eps_2,b)$ such that the following holds for each $n\geq d\geq d_0$. Suppose that $G$ is an $n$-vertex $Q_3$-free bipartite $(\eps_1,\eps_2 d)$-expander with $\delta(G)\ge d$.
	
	Suppose $\log^{10} n\leq D\leq \log^{b}n$, and $U\subseteq V(G)$ with $|U|\leq D/2\log^3n$, and let $m=\frac{8000}{\eps_1}\log^3n$. Suppose $F_1,F_2\subseteq G-U$ are vertex-disjoint such that $F_i$ is a $(D,m)$-expansion of $v_i$, for each $i\in[2]$. Let $\log^{7}n\leq \ell\leq n/\log^{10}n$ be such that $\ell=\pi(v_1,v_2,G)\mod 2$.
	
	Then, there is a $v_1,v_2$-path with length $\ell$ in $G-U$.
\end{lemma}

The careful readers might notice that the original Lemma 4.8 in~\cite{LiuMont20} requires the graph $G$ to not contain a large clique subdivision ($\mathsf{TK}_{d/2}^{(2)}$-free). The above version for $Q_3$-free graphs~$G$ can be proved by following the proof in~\cite{LiuMont20} and replacing the use of Proposition~3.16 there with Proposition~\ref{prop-Q3} here.

	
\section{Main lemmas}\label{sec-main-proof}

To prove Theorem~\ref{thm-thomassen-conj}, we first find two copies of krakens whose cycles are of the same length. Then we link sequentially and disjointly the legs of one kraken to those of the other one. We package these two steps into the following two lemmas, respectively.

\smallskip

The first lemma is the key one, which constructs  a kraken robustly in an expander.

\begin{lemma}\label{lem-robust-kraken}
	For each $0<\eps_1,\eps_2<1$ and integer $b\ge 10$, there exists $d_0=d_0(\eps_1,\eps_2,b)$ such that the following holds for each $n\ge d\ge d_0$.
	
	Let $G$ be a $Q_3$-free $n$-vertex bipartite $(\eps_1,\eps_2 d)$-expander with $\delta(G)\geq d$. Let $L$ be the set of vertices with degree at least $e^{(\log\log n)^2}$. Let $m=200\eps^{-1}\log^3n$ and let $U\subseteq V(G)$ satisfy $|U|\leq (\log n)^{2b}$. 
	
	Then, for some $k\leq\log n$, $G-U$ contains a $(k,2m,(\log n)^{b})$-kraken $(C, u_j, F_j, P_j)$, $j\in[k]$,  such that 
	\begin{itemize}
	    \item for each $j\in [k]$, either $u_j\in L$ or $F_j\subseteq G-L$; and
	    
	    \item any distinct legs $F_{j}, F_{j'}$ in $G-L$ are a distance at least $(\log n)^{1/10}$ apart from each other and from $U\setminus L$ in $G-L$.
	\end{itemize}
\end{lemma}

The next lemma carries out the finishing blow. It allows us to link each pair of vertices in the cycles inside krakens via their legs disjointly to construct a pillar. Lemma~\ref{lem-finalconnect} kicks in here to make sure that all the paths used are of the same length.

\begin{lemma}\label{lem-link-2pulp-adj}
	For each $0<\eps_1,\eps_2<1$ and $t\ge 10$, there exists $d_0=d_0(\eps_1,\eps_2,t)$ such that the following holds for each $n\ge d\ge d_0$.
	
	Let $G$ be a $Q_3$-free $n$-vertex bipartite $(\eps_1,\eps_2 d)$-expander with $\delta(G)\geq d$ and $L$ be the set of vertices with degree at least $e^{(\log\log n)^2}$. Let $m=400\eps^{-1}\log^3n$ and let $\mathsf{K}_\alpha=(C_\alpha,u_j^\alpha,F_j^\alpha,P_j^\alpha)$ and $\mathsf{K}_\beta=(C_\beta,u_j^\beta,F_j^\beta,P_j^\beta)$, $j\in[s]$, be two disjoint $(s,m,(\log n)^{2t})$-krakens, for some $s\leq\log n$, with $V(C_\alpha)=\{v_1^\alpha,\dots,v_s^\alpha\}$ and $V(C_\beta)=\{v_1^\alpha,\dots,v_s^\alpha\}$, and such that 
	\begin{itemize}
	    \item for each $\sigma\in\{\alpha,\beta\}$ and $j\in [s]$, either $u_j^\sigma\in L$ or $F_j^\sigma\subseteq G-L$; and
	    
	    \item all legs in $K_{\alpha}$ and $K_{\beta}$ lying completely in $G-L$ are a distance at least $(\log n)^{1/10}$ apart from each other in $G-L$.
	\end{itemize}
	Then, for any $\log^{7}n\leq \ell\leq \log^{t}n$ with $\ell=\pi(v_1^{\alpha},v_1^{\beta},G)$, there is a collection of pairwise disjoint paths $Q_i$, $i\in[s]$, such that each $Q_i$ is a $v_i^\alpha,v_i^\beta$-path of length $\ell$ internally disjoint from $C_\alpha$ and $C_\beta$.
\end{lemma}

\begin{proof}[Proof of Theorem~\ref{thm-thomassen-conj}]
	By passing to a subgraph using Theorem~\ref{thm-pass-to-expander}, we may assume that we start with an $n$-vertex bipartite $(\eps_1,\eps_2d)$-expander graph with minimum degree $d$. As a 3-dimensional cube $Q_3$ is a pillar, we may further assume that $G$ is $Q_3$-free. We shall repeatedly apply Lemma~\ref{lem-robust-kraken} to obtain two krakens whose cycles are of the same length and then invoke Lemma~\ref{lem-link-2pulp-adj} to finish the proof.
	
	More precisely, let $t\ge 10$. For $i<\log n+1$, suppose we have already found $i$ disjoint $(k_i,m,(\log n)^{2t})$-krakens for some $k_i\le \log n$. Let $U$ be the union of the vertex sets of all these krakens, then, as $t\ge 10$, $|U|\leq i\cdot \log n(1+10m+(\log n)^{2t})\leq (\log n)^{4t}$. Then by Lemma~\ref{lem-robust-kraken}, we can find another kraken in $G-U$. 
	
	Thus, we can find at least $\log n+1$ disjoint krakens. As the cycle in each kraken has length at most $\log n$, by pigeonhole principle,
	among these krakens there are two whose cycles are of the same length $s$, for some $s\leq\log n$. 
	Let $L$ be the set of vertices with degree at least $e^{(\log\log n)^2}$.
	Lemma~\ref{lem-robust-kraken} and the choice of $U$ guarantee that the legs of these two krakens not containing high degree vertices from $L$ can be taken far apart from each other in $G-L$. Thus, Lemma~\ref{lem-link-2pulp-adj} applies and we can link these two krakens to obtain the desired pillar.
\end{proof}


\section{Sustainable kraken fishing}\label{sec-robust-kraken}
In this section, we prove Lemma~\ref{lem-robust-kraken}, which constructs a kraken in an expander robustly. Let us first describe the high level idea for the proof.

Suppose, for contradiction, that there is no kraken in $G-U$ with the required size and properties. Take a maximal collection of krakens $\mathbf{K}$ in $G-U-L$ so that their legs are far apart. Our goal is to show that we can eventually expand all the legs of one of these krakens to obtain a desired kraken, giving a contradiction.

To this end, we first show that there are many krakens in $\mathbf{K}$ (see Claim~\ref{claim-size-P0}). This can be done by repeatedly applying Lemma~\ref{lem-kraken-from-kraken} to some appropriate expander subgraph in $G-U-L$. We then consider maximal collections $\mathcal{P}$ and $\mathcal{Q}$ of paths from legs of krakens in $\mathbf{K}$ to either~$L\setminus U$ or to $\mathcal{Z}$, a collection of large sets  each with small diameter. We show that there cannot be a kraken in $\mathbf{K}$ with all its legs linked to paths in $\mathcal{P}\cup\mathcal{Q}$ (see Claim~\ref{claim-free-legs}), i.e., all the krakens in $\mathbf{K}$ have at least one `free' leg that is not linked to $\mathcal{P}\cup\mathcal{Q}$, as otherwise we can extend the legs of the kraken using the large sets in $\cZ$ or the neighbourhoods of the large degree vertices in $L\setminus U$ to obtain a large kraken in $G-U$.

Finally, we collectively expand all the free legs in all the krakens in $\mathbf{K}$. Then by Lemma~\ref{lem-expand-together}, one free leg of one of these krakens must expand and can therefore be linked to an unused set in $\mathcal{Z}$, contradicting the maximality of $\mathcal{Q}$. This will conclude the proof.

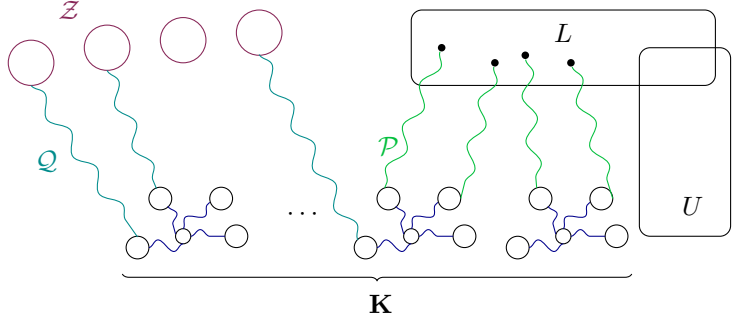
\begin{figure}[h]
	\centering
	\begin{tikzpicture}
		\draw[rounded corners] (2,2) rectangle (6,3);
		\draw[rounded corners] (5,0) rectangle (6.2,2.5);
		\node[inner sep= 1pt](t1) at (4,2.7){\small$L$};
		\node[inner sep= 1pt](t2) at (5.7,0.4){\small$U$};
		\draw[snake=brace,gap around snake=5mm] (5.4,-0.5)--(-2.3,-0.5);
		\node[inner sep= 1pt](t7) at (1.6,-0.9){\small$\mathbf{K}$};
		\draw[black] (4,0) circle (0.1cm);
		\draw[black] (4.7,0) circle (0.15cm);
		\draw[decorate, decoration=snake, segment length=4mm,darkblue] (4.1,0) -- (4.55,0);
		\draw[black] (4.5,0.5) circle (0.15cm);
		\draw[decorate, decoration=snake, segment length=4mm,darkblue] (4,0.1) -- (4.35,0.45);
		\draw[black] (3.7,0.5) circle (0.15cm);
		\draw[decorate, decoration=snake, segment length=4mm,darkblue] (3.8,0.4) -- (3.9,0.03);
		\draw[black] (3.4,-0.15) circle (0.15cm);
		\draw[decorate, decoration=snake, segment length=4mm,darkblue] (3.55,-0.15) -- (4,-0.1);
		\draw[black] (2,0) circle (0.1cm);
		\draw[black] (2.7,0) circle (0.15cm);
		\draw[decorate, decoration=snake, segment length=4mm,darkblue] (2.1,0) -- (2.55,0);
		\draw[black] (2.5,0.5) circle (0.15cm);
		\draw[decorate, decoration=snake, segment length=4mm,darkblue] (2,0.1) -- (2.35,0.45);
		\draw[black] (1.7,0.5) circle (0.15cm);
		\draw[decorate, decoration=snake, segment length=4mm,darkblue] (1.8,0.4) -- (1.9,0.03);
		\draw[black] (1.4,-0.15) circle (0.15cm);
		\draw[decorate, decoration=snake, segment length=4mm,darkblue] (1.55,-0.15) -- (2,-0.1);
		\draw[black] (-1,0) circle (0.1cm);
		\draw[black] (-0.3,0) circle (0.15cm);
		\draw[decorate, decoration=snake, segment length=4mm,darkblue] (-0.9,0) -- (-0.45,0);
		\draw[black] (-0.5,0.5) circle (0.15cm);
		\draw[decorate, decoration=snake, segment length=4mm,darkblue] (-1,0.1) -- (-0.65,0.45);
		\draw[black] (-1.3,0.5) circle (0.15cm);
		\draw[decorate, decoration=snake, segment length=4mm,darkblue] (-1.2,0.4) -- (-1.1,0.03);
		\draw[black] (-1.6,-0.15) circle (0.15cm);
		\draw[decorate, decoration=snake, segment length=4mm,darkblue] (-1.45,-0.15) -- (-1,-0.1);
		\node[inner sep= 1pt](t3) at (0.6,0.3){\small$\cdots$};
		\draw[darkraspberry] (-2,2.5) circle (0.3cm);
		\draw[darkraspberry] (-3,2.3) circle (0.3cm);
		\draw[darkraspberry] (-1,2.6) circle (0.3cm);
		\draw[darkraspberry] (0,2.7) circle (0.3cm);
		\node[inner sep= 1pt,darkraspberry](t4) at (-2.5,3){\small$\mathcal{Z}$};
		\node[inner sep= 1pt](l1) at (2.4,2.5)[circle,fill]{};
		\node[inner sep= 1pt](l3) at (3.1,2.3)[circle,fill]{};
		\node[inner sep= 1pt](l4) at (3.5,2.4)[circle,fill]{};
		\node[inner sep= 1pt](l5) at (4.1,2.3)[circle,fill]{};
		\draw[decorate, decoration=snake, segment length=6mm,darkpastelgreen] (l5) -- (4.65,0.5);
		\draw[decorate, decoration=snake, segment length=6mm,darkpastelgreen] (l4) -- (3.65,0.64);
		\draw[decorate, decoration=snake, segment length=6mm,darkpastelgreen] (l3) -- (2.65,0.5);
		\draw[decorate, decoration=snake, segment length=6mm,darkpastelgreen] (l1) -- (1.65,0.64);
		\node[inner sep= 1pt,darkpastelgreen](t6) at (1.7,1.2){\small$\mathcal{P}$};
		\draw[decorate, decoration=snake, segment length=6mm,darkcyan] (0,2.41) -- (1.3,-0.05);
		\draw[decorate, decoration=snake, segment length=6mm,darkcyan] (-3,2) -- (-1.6,-0.01);
		\draw[decorate, decoration=snake, segment length=6mm,darkcyan] (-2,2.2) -- (-1.35,0.63);
		\node[inner sep= 1pt,darkcyan](t4) at (-2.8,1){\small$\mathcal{Q}$};
	\end{tikzpicture}
\caption{An illustration of the proof of Lemma~\ref{lem-robust-kraken}.}
\end{figure}

\begin{proof}[Proof of Lemma~\ref{lem-robust-kraken}]
	Suppose, for contradiction, that, for any $k\leq\log n$, $G-U$ contains no $(k,2m,(\log n)^{b})$-kraken with the desired properties. Set 
	$$\ell_0=(\log\log n)^{20}, \quad \Delta=e^{(\log\log n)^2},  \quad\text{ and } \quad G':=G-L.$$ 
	Thus, $\Delta(G')\leq \Delta$. Further define 
	$$U_0=\{v\in V(G)\setminus U: d_G(v,U)\geq d/2\}.$$ 
	Note that, if $|U_0|\geq (\log n)^{6b}\geq |U|^3$, then, by Proposition~\ref{prop-Q3} with $(U,W)_{\ref{prop-Q3}}=(U_0,U)$, $G$ contains a copy of $Q_3$, a contradiction. Therefore, we can assume that $|U_0|\leq(\log n)^{6b}$, and hence, as $\delta(G)\geq d$ and $n\geq d_0(\eps_1,\eps_2,b)$ is large, $G-U$ contains at least 
	$$\frac{1}{2}\sum_{v\not\in U\cup U_0}d_{G-U}(v)\ge \Big(n-|U|-|U_0|\Big)\cdot\frac{d}{4}\geq \Big(n-(\log n)^{2b}-(\log n)^{6b}\Big)\cdot\frac{d}{4}\geq \frac{nd}{8}$$
	edges. Let $U_1:=U\cup U_0$, we have that $|U_1|\leq (\log n)^{2b}+(\log n)^{6b}\leq 2(\log n)^{6b}$.
	
	\smallskip
	
	Take a maximal collection $\mathbf{K}$ of krakens in $G-U$ such that the following hold:
	\stepcounter{propcounter}
	\begin{enumerate}[label = {\bfseries \Alph{propcounter}\arabic{enumi}}]
		\item The legs of each kraken $\mathsf{K}=(C,u_j,F_j,P_j)\in\mathbf{K}$ that do not contain a vertex in $L$ lie in $G'$ and are at least $10\ell_0$-apart from the legs of other krakens and from $U_1\setminus L$ in $G'$.\label{pulp1}
		\item We can index krakens in $\mathbf{K}$ so that each $\mathsf{K}_H\in\mathbf{K}$ is a $(k_H,m_H,(\log n_H)^{b})$-kraken, for some $n_H,m_H,k_H\in\mathbb{N}$ with $\frac{d}{64}\le n_H\le n$, $m_H=200\eps^{-1}\log^3n_H \leq m$ and $k_H\leq\log n_H$.\label{pulp2}
	\end{enumerate}

\begin{claim}\label{claim-size-P0}
	$|\mathbf{K}|\geq n^{1/8}$.
\end{claim}
\begin{poc}
	Suppose, for contradiction, that $|\mathbf{K}|< n^{1/8}$. Let~$W=\big(U_1\cup (\cup_{\mathsf{K}\in\mathbf{K}}V(\mathsf{K}))\big)\setminus~L$. Note that, for each $\mathsf{K}_H\in\mathbf{K}$, we have
	$$|V(\mathsf{K}_H)|\leq \sum_{j=1}^{k_H}|V(F_j)|+\sum_{j=1}^{k_H}|V(P_j)|\leq\log n_H\cdot((\log n_H)^{b}+10m_H)\leq 20(\log n_H)^{b+1},$$
	which implies that $|W|\leq 2(\log n)^{6b}+n^{1/8}\cdot20(\log n_H)^{b+1}\leq n^{1/7}$. Let $W'=B_{G'}^{10\ell_0}(W)$. Since $\Delta(G')\leq e^{(\log\log n)^2}$, then $|W'|\leq 2|W|\Delta^{10\ell_0}\leq n^{1/6}$.
	
	Thus, there are at most $|W'|\Delta\leq\Delta\cdot n^{1/6}\leq nd/16$ edges in $G$ with some vertex in~$W'$. As $G-U$ contains at least $nd/8$ edges, $G-U-W'$ contains at least $nd/16$ edges. Consequently, $d(G-U-W')\geq d/8$. Then, by Theorem~\ref{thm-pass-to-expander}, $G-U-W'$ contains a bipartite $(\eps_1,\eps_2d/64)$-expander~$H$ with $\delta(H)\geq d/64$.
	
	Thus, by Lemma~\ref{lem-kraken-from-kraken} with $n_{\ref{lem-kraken-from-kraken}}=n_H:=|H|$,  $m_{\ref{lem-kraken-from-kraken}}=m_H=200\eps_1^{-1}\log^3n_H\leq m$ and some $k_H\leq\log n_H$, there exists a $(k_H,m_H,\log^{b}n_H)$-kraken in $H$, which is at distance at least $10\ell_0$-apart from every other kraken in $\mathbf{K}$ and from $U_1\setminus L$ in $G'$, as $H\subseteq G-U-W'$, a contradiction to the maximality of~$\mathbf{K}$.	
\end{poc}

Now, let $p_0=|\mathbf{K}|$ and write $\mathsf{K}_i$ for a $(k_{H_i},m_{H_i},(\log n_{H_i})^{b})$-kraken in $\mathbf{K}$, for each $i\in[p_0]$. We may assume $p_0=n^{1/8}$.
Moreover, for each $i\in[p_0]$, write $c_i\leq \log n_{H_i}$ for the length of the cycle $\bar{C_i}$ in $\mathsf{K}_i$, and denote by $F_{i,j}$, $u_{i,j}$, $P_{i,j}$, $j\in [c_i]$, the $j$-th leg, its end and the corresponding path of $\mathsf{K}_i$, respectively; that is, for each $i\in[p_0]$, $\mathsf{K}_i=(\bar{C_i},u_{i,j},F_{i,j},P_{i,j})$, $j\in[c_i]$.

\begin{claim}\label{claim-balls}
	 There exists a collection of connected set of vertices $\cZ:=\{Z_i:i\in[m^2]\}$ in $G'-U-V(\mathbf{K})$ such that for each $i\in[m^2]$, $Z_i$ has size $(\log n)^{100b}$ with diameter at most $m$ and it is at distance at least $(\log n)^{1/10}$ apart from other sets in $\cZ$ and from $U$ in $G'$.
\end{claim}
\begin{poc}
	Take a maximal collection of sets $Z_i$, $i\in[s]$, with the claimed size, which are pairwise far apart and far from $U$ in $G'$. If $s<m^2$, then
	$$\big|B_{G'}^{(\log n)^{1/10}}(\cup_{i\in[s]}Z_i\cup U)\big|\leq 2\cdot \Big(s\cdot  (\log n)^{100b}+(\log n)^{2b}\Big)\cdot \Delta(G')^{(\log n)^{1/10}} <\sqrt{n}.$$
	Thus, by Lemma~\ref{lem-find-large-ball}	with $W_{\ref{lem-find-large-ball}}=V(\mathbf{K})\cup B_{G'}^{(\log n)^{1/10}}(\cup_{i\in[s]}Z_i\cup U)$, we can find another large set with small diameter  in $G'-U-V(\mathbf{K})$ far apart from $\cup_{i\in[s]}Z_i\cup U$ in $G'$, a contradiction to the maximality of $s$.
\end{poc}

Take now two collections of paths $\cP$ and $\cQ$ as follows.
\stepcounter{propcounter}
\begin{enumerate}[label = {\bfseries \Alph{propcounter}\arabic{enumi}}]
	\item\label{whatlunch1} Let $\cP$ be a maximal collection of paths in $G-U$ from legs $V(F_{i,j})$, $j\in[c_i],i\in [p_0]$ of krakens in~$\mathbf{K}$ to $L\setminus U$ with length at most $\ell_0$ each and such that paths from the same kraken are vertex-disjoint.
	
	Subject to $|\cP|$ being maximal, let $\ell(\cP):=\sum_{P\in\mathcal{P}}\ell(P)$ be minimised.
	
	\item\label{whatlunch2} Let $\cQ$ be a maximal collection of paths in $G-U$ from legs $V(F_{i,j})$, $j\in[c_i],i\in [p_0]$ of krakens in $\mathbf{K}$ that are not already linked to paths in $\cP$  to $\cZ$ with length at most $3m$ each and such that:
	\begin{itemize}\itemsep=0pt
		\item each $Z_i\in\cZ$ is linked to at most one leg of the same kraken; and
		\item paths in $\cP\cup\cQ$ from the same kraken are pairwise vertex-disjoint.
	\end{itemize} 
	Subject to $|\cQ|$ being maximal, let $\ell(\cQ):=\sum_{Q\in\mathcal{Q}}\ell(Q)$ be minimised.
\end{enumerate}

\begin{claim}\label{claim-free-legs}
	There is no kraken in $\mathbf{K}$ with all its legs linked to paths in $\cP\cup\cQ$.
\end{claim}
\begin{poc}
	Suppose, for contradiction, that there exists $\mathsf{K}_i\in\mathbf{K}$ that has all its legs linked to paths in $\cP\cup\cQ$. Let $p=|V(\cP)\cap\bar{C_i}|$ and $q=|V(\cQ)\cap\bar{C_i}|$, and note that $p+q=c_i\leq \log n$. Let $v_1,\ldots,v_{p+q}$ be vertices in $\bar{C_i}$.
	
	Relabelled if necessary, we may assume that $v_1,\dots,v_p$ are the vertices in $\bar{C_i}$ whose corresponding legs are linked by paths in $\mathcal{P}$ to high degree vertices $u_1,\dots,u_p\in V(\cP)\cap(L\setminus U)$, and $Z_{p+1},\dots,Z_{p+q}$ are the sets in $\cZ$ linked to $\mathsf{K}_i$ by paths in $\cQ$. For $j\in[p+1,p+q]$, let~$u_j$ be the vertex in $Z_j\cap V(\cQ)$. Since $|Z_j|=(\log n)^{100b}$ for every $j\in[p+1,p+q]$, using Proposition~\ref{prop:trimming}, we can take connected sets $\widetilde{Z}_j\subseteq Z_j$ with $|\widetilde{Z}_j|=\log^{b}n$ and $u_j\in \widetilde{Z}_j$ such that all $\widetilde{Z}_j$ are pairwise disjoint.
	
	On the other hand, for each $j\in[p]$, we can choose pairwise disjoint sets $Y_j\subseteq N(u_j)$ with $|Y_j|=\log^{b}n$ that are also disjoint from $\{\widetilde{Z}_j:j\in[p+1,p+q]\}$, as $|N(u_j)|\geq\Delta$.
	
	For each $P_j\in\cP$, $j\in[p]$, we can extend it to a $v_j,u_j$-path $\widetilde{P}_{i,j}\subseteq P_j\cup F_{i,j}\cup P_{i,j}$ with $|\widetilde{P}_{i,j}|\leq \ell_0+m_{H_i}+10m_{H_i}\leq 20m$. Similarly, for each $Q_j\in\cQ$, $j\in[p+1,p+q]$, we can extend it to a $v_j,u_j$-path $\widetilde{Q}_{i,j}\subseteq Q_j\cup F_{i,j}\cup P_{i,j}$ with length $|\widetilde{Q}_{i,j}|\leq 3m+m_{H_i}+10m_{H_i}\leq 20m$.
	
	Thus, we get a kraken $\mathsf{K}_i'=(\bar{C_i},u_j,F_{i,j}',P_{i,j}')$, $j\in[c_i]$, where, for $j\in[p]$, $F_{i,j}'=G[Y_j]\cup\{u_j\}$ and $P_{i,j}'=\widetilde{P}_{i,j}$ and for $j\in[p+1,p+q]$, $F_{i,j}'=G[\widetilde{Z}_j]$ and $P_{i,j}'=\widetilde{Q}_{i,j}$, which is a $(c_i,2m,(\log n)^{b})$-kraken in $G-U$. Note that by \ref{pulp1} and the choice of $\cZ$, $\mathsf{K}_i'$ has the desired property that the legs whose ends are not in $L$ lie completely in $G'$ and are far apart from each other and from $U_1\setminus L$ in $G'$, a contradiction to our initial assumption.
\end{poc}

Therefore, for each $\mathsf{K}_i$, $i\in[p_0]$, there must exist one `free' leg $F_{i,j_0}$ which is not linked to any path in $\mathcal{P}\cup\cQ$. Note that by definition, for every $i\in[p_0]$, $F_{i,j_0}\subseteq G'$, as otherwise, if~$F_{i,j_0}$ contains some vertex $u\in L$, then we can view $\{u\}$ as a single-vertex path from $F_{i,j_0}$ to~$L$, a contradiction to the maximality of $\cP$ in \ref{whatlunch1} and $F_{i,j_0}$ being free.

Now, we shall use Lemma~\ref{lem-expand-together} to collectively expand free legs in all krakens to find one that expands well. If succeeded, we can then link this free leg to an unused set in $\cZ$ to reach the final contradiction to the maximality of $\cQ$.  To this end, we need to specify the sets $A_i,B_i,C_i$ to invoke Lemma~\ref{lem-expand-together}. Let $\mathcal{P}_i$ and $\cQ_i$ be the subcollections of paths in $\mathcal{P}$ and~$\cQ$, respectively, linked to the kraken~$\mathsf{K}_i$, and for each $i\in[p_0]$, define the sets
$$A_i:=F_{i,j_0}, \quad B_i:=\bar{C_i}\cup\Big(\bigcup_{j\in[c_i]}P_{i,j}\Big)\setminus \{u_{i,j_0}\}, \quad\text{and }\quad C_i:=V(\cP_i\cup\cQ_i).$$

\begin{claim}
	There is some $\ell\in[p_0]$ for which
	$$|B^{\ell_0}_{G-U-B_\ell-C_\ell}(A_\ell)|\geq(\log n)^{200b}.$$
\end{claim}
\begin{poc}
	In order to prove this claim, we have to show that conditions \emph{\ref{exptog}}--\emph{\ref{exptog5}} hold with $A_i,B_i,C_i$ as above. First, by \ref{pulp2}, $|A_i|=(\log n_{H_i})^{b}\geq(\log (d/64))^{b}$ can be taken suffiently large by taking $d\ge d_0(\eps_1,\eps_2,b)$ large, thus \emph{\ref{exptog}} holds.
	
	It is clear from the choice of $A_i$ that it is disjoint from $B_i\cup C_i$ in $V(G)\setminus U$. Moreover, as $b\ge 10$, again by \ref{pulp2}, we have
	$$|B_i|\leq\log n_{H_i}+\log n_{H_i}\cdot10m_{H_i}\leq (\log n_{H_i})^{6}\leq\frac{|A_i|}{\log^{10}|A_i|},$$
	so that \emph{\ref{exptog2}} holds.
	
	For \emph{\ref{exptog3}}, note that for any path $P\in\cP_i\cup\cQ_i$ and any $r\in\mathbb{N}$, we have
	$$|N_{G-U-B_i}(B_{G-U-B_i-C_i}^{r-1}(A_i))\cap V(P)|\leq  r+1.$$
	Indeed, if this is not the case, we can replace the initial segment of at least $r+1$ vertices in~$P$ by the length-$r$ path in $B_{G-U-B_i}(B_{G-U-B_i-C_i}^{r-1}(A_i))$ to obtain a new path, shorter than~$P$, from $A_i$ to the endvertex of $P$ in $L\setminus U$ or $\cZ$. This contradicts the minimality of $\ell(\cP)$ or~$\ell(\cQ)$ in \ref{whatlunch1} and \ref{whatlunch2}.
	
	Thus, as $|\cP_i\cup \cQ_i|\le c_i\le \log n_{H_i}\le |A_i|^{1/4}$, we have that
	$$|N_{G-U-B_i}(B_{G-U-B_i-C_i}^{r-1}(A_i))\cap C_i|\leq\log n_{H_i}\cdot (r+1)\leq\sqrt{|A_i|}\cdot r.$$
	That is, $C_i$ is $(\sqrt{|A_i|},1)$-thin around $A_i$ in $G-U-B_i$, yielding~\emph{\ref{exptog3}}.
	
	Note that there is no path with length at most $\ell_0$ from $A_i$ to $L\setminus U$ in $G-U-B_i-C_i$, as otherwise it would be a contradiction to the maximality of $\mathcal{P}$. Thus, we have that $B^{\ell_0}_{G-U-B_i-C_i}(A_i)=B^{\ell_0}_{G-L-U-B_i-C_i}(A_i)$, which, by \ref{pulp1}, is disjoint from $U_1$, and by the choice of $U_0\subseteq U_1$, we have that \emph{\ref{exptog4}} holds.
	
	Similarly, for any $j\in[p_0]\setminus\{i\}$, we have that $B^{\ell_0}_{G-U-B_j-C_j}(A_j)=B^{\ell_0}_{G-L-U-B_j-C_j}(A_j)$. So by \ref{pulp1}, $B^{\ell_0}_{G-U-B_j-C_j}(A_j)$ and $B^{\ell_0}_{G-U-B_i-C_i}(A_i)$ are disjoint. In particular, $A_i$ and $A_j$ are at distance at least $2\ell_0$ apart in $G-U-B_i-C_i-B_j-C_j$, and hence \emph{\ref{exptog5}} holds.
	
	Therefore, Lemma~\ref{lem-expand-together} applies and so there is some $\ell\in[p_0]$ for which
	$$|B^{\ell_0}_{G-U-B_\ell-C_\ell}(A_\ell)|\geq(\log n)^{200b},$$
	as claimed.
\end{poc}

Partition $\cZ=\mathcal{B}\cup\mathcal{B}'$, where $\mathcal{B}'$ is the subcollection of sets in $\cZ$ that are linked through paths in~$\cQ$ to the kraken $\mathsf{K}_\ell$, and  $\mathcal{B}=\cZ\setminus\mathcal{B}'$. We now show that there is a path of length at most $3m$ from $A_\ell$ to $\mathcal{B}$ avoiding $V(\mathcal{B}')\cup U\cup B_\ell\cup C_\ell$, which will lead us to a contradiction to the maximality of $\mathcal{Q}$, and will complete the proof. For this, we observe that
$$|V(\mathcal{B}')\cup U\cup B_\ell\cup C_\ell|\leq\log n_{H_{\ell}}\cdot(\log n)^{100b}+(\log n)^{2b}+(\log n_{H_{\ell}})^{6}+\log n_{H_{\ell}}\cdot 3m\leq 2(\log n)^{100b+1}.$$
Therefore, we have that 
$$10|B^{\ell_0}_{G-U-B_\ell-C_\ell}(A_\ell)|\geq10(\log n)^{200b}\geq\log^3n\cdot|V(\mathcal{B}')\cup U\cup B_\ell\cup C_\ell|$$ 
and 
$$10|V(\mathcal{B})|\geq10(m^2-\log n)(\log n)^{100b}\geq\log^3n\cdot|V(\mathcal{B}')\cup U\cup B_\ell\cup C_\ell|.$$
Then applying Lemma~\ref{lem-short-diam-new} with $(A,B,W)_{\ref{lem-short-diam-new}}=(B^{\ell_0}_{G-U-B_\ell-C_\ell}(A_\ell),V(\mathcal{B}),V(\mathcal{B}')\cup U\cup B_\ell\cup C_\ell)$, we get a path from $B^{\ell_0}_{G-U-B_\ell-C_\ell}(A_\ell)$ to $\mathcal{B}$ of length at most $40\eps_1^{-1}\log^3n$, which can be extended to an $A_\ell,\mathcal{B}$-path of length at most $\ell_0+40\eps_1^{-1}\log^3n\leq3m$, contradicting the maximality of $\mathcal{Q}$, and completing the proof. 
\end{proof}

\section{Krakens hand in hand}\label{sec-link-2krakens}
\begin{proof}[Proof of Lemma~\ref{lem-link-2pulp-adj}]
	Let $\mathsf{K}_\alpha,\mathsf{K}_\beta$ be as given and fix $\log^{7}n\leq \ell\leq \log^{t}n$ with $\ell=\pi(v_1^{\alpha},v_1^{\beta},G)$. Set $\ell_0=(\log\log n)^{20}$ and $\Delta=e^{(\log\log n)^2}$. Since $G$ is bipartite, we know that $s$ is even. Without loss of generality, we may assume that $v_1^{\alpha},v_3^{\alpha},\dots,v_{s-1}^{\alpha}$, and $v_1^{\beta},v_3^{\beta},\dots,v_{s-1}^{\beta}$ lie in one side of the bipartition of $G$, while $v_2^{\alpha},v_4^{\alpha},\dots,v_s^{\alpha}$, and $v_2^{\beta},v_4^{\beta},\dots,v_s^{\beta}$ lie in the other side. In particular, $\ell$, the length of each path $Q_i$ to be constructed, is even. We call a leg of $\mathsf{K}_\alpha$ or $\mathsf{K}_\beta$ a low degree leg if it lies completely in $G-L$ and high degree leg otherwise.
	
	We sequentially find pairwise disjoint $v_j^{\alpha},v_j^{\beta}$-path $Q_j$ such that, for each $j\in[s]$, $Q_j$ is disjoint from $\bigcup_{i=j+1}^s V(P_{i}^\alpha\cup F_{i}^{\alpha}\cup P_{i}^\beta\cup F_{i}^{\beta})$, as follows. Suppose we have already constructed $Q_1,Q_2,\dots,Q_k$, and we want to embed $Q_{k+1}$, for some $k\in\{0,1,\dots,s-1\}$. We do so by connecting the legs $F_{k+1}^\alpha$ and $F_{k+1}^\beta$. For $\sigma\in\{\alpha,\beta\}$, if $F_{k+1}^\sigma$ is a low degree leg, then we first expand it. 
	
	To be precise, we consider the following three cases. See Figure~\ref{fig-link-kraken}.
	
	\begin{figure}[h]
		\centering
		\begin{tikzpicture}
		\draw[black] (2,0) circle (0.5cm);
		\draw[black] (0,-3) circle (0.4cm);
		\draw[black] (2.5,-3) circle (0.4cm);
		\draw[black] (4.5,-0.5) circle (0.4cm);
		\node[inner sep= 1pt](a1) at (1.55,-0.2)[circle,fill]{};
		\node[inner sep= 1pt](a2) at (2.3,-0.4)[circle,fill]{};
		\node[inner sep= 1pt](a3) at (2.5,-0.1)[circle,fill]{};
		\node[inner sep= 1pt](a4) at (4.1,-0.4)[circle,fill]{};
		\node[inner sep= 1pt](a5) at (5.5,-0.5)[circle,fill]{};
		\node[inner sep= 1pt](a6) at (2.45,0.2)[circle,fill]{};
		\node[inner sep= 1pt](a7) at (2.1,0.48)[circle,fill]{};
		\draw[decorate, decoration=snake, segment length=6mm,darkpastelgreen] (a1) -- (0,-2.6);
		\draw[decorate, decoration=snake, segment length=6mm,darkpastelgreen] (a2) -- (2.5,-2.6);
		\draw[decorate, decoration=snake, segment length=6mm,darkpastelgreen] (a3) -- (a4);
		\draw[decorate, decoration=coil, segment length=8mm,red] (a4) -- (a5);
		\draw[red] (a5) -- (6.2,0);
		\draw[red] (a5) -- (6.2,-1);
		\draw[red] (6.2,-1) -- (6.2,0);
		\draw[black] (12,0) circle (0.5cm);
		\draw[black] (11.5,-3) circle (0.4cm);
		\draw[black] (13.5,-3) circle (0.4cm);
		\draw[black] (9.5,-1) circle (0.4cm);
		\node[inner sep= 1pt](b1) at (11.8,-0.45)[circle,fill]{};
		\node[inner sep= 1pt](b2) at (12.3,-0.4)[circle,fill]{};
		\node[inner sep= 1pt](b3) at (11.5,-0.1)[circle,fill]{};
		\node[inner sep= 1pt](b4) at (11.55,0.2)[circle,fill]{};
		\node[inner sep= 1pt](b5) at (12.1,0.48)[circle,fill]{};
		\draw[decorate, decoration=snake, segment length=6mm,darkpastelgreen] (b1) -- (11.5,-2.6);
		\draw[decorate, decoration=snake, segment length=6mm,darkpastelgreen] (b2) -- (13.5,-2.6);
		\draw[decorate, decoration=snake, segment length=6mm,darkpastelgreen] (b3) -- (9.9,-0.9);
		\draw[deepcarmine, thick] (9.3,-0.65) -- (8.3,-0.3);
		\draw[deepcarmine, thick] (9.4,-1.4) -- (8.3,-1.7);
		\draw[deepcarmine, thick] (8.3,-0.3) -- (8.3,-1.7);
		\draw[decorate, decoration=snake, segment length=18mm,darkblue] (a6) .. controls (8,0.5) .. (b4);
		\draw[decorate, decoration=snake, segment length=18mm,darkblue] (a7) .. controls (7,1.5) .. (b5);
		\node[inner sep= 1pt,black](t1) at (7,1){\small$\vdots$};
		\node[inner sep= 1pt,darkblue](t2) at (7.5,1.6){\small$Q_1$};
		\node[inner sep= 1pt,darkblue](t3) at (7.5,0.6){\small$Q_k$};
		\node[inner sep= 1pt,black](t4) at (1.1,0.3){\small$C_\alpha$};
		\node[inner sep= 1pt,black](t5) at (12.9,0.3){\small$C_\beta$};
		\node[inner sep= 1pt,darkpastelgreen](t6) at (3.3,-0.6){\small$P_{k+1}^\alpha$};
		\node[inner sep= 1pt,black](t7) at (4.3,-1.2){\small$F_{k+1}^\alpha$};
		\node[inner sep= 0.5pt,black](t8) at (3.9,-0.1){\small$u_{k+1}^\alpha$};
		\node[inner sep= 0.3pt,black](t9) at (5.4,-0.1){\small$w_{k+1}^\alpha$};
		\node[inner sep= 0.5pt,red](t10) at (5.3,-0.8){\small$R_{k+1}^\alpha$};
		\node[inner sep= 0.5pt,red](t11) at (4.4,-1.8){\small Case $3$};
		\node[inner sep= 1pt,black](t12) at (3.4,-3){\small$F_{k+2}^\alpha$};
		\node[inner sep= 1pt,black](t13) at (1.2,-3){\small$\cdots$};
		\node[inner sep= 1pt,deepcarmine](t14) at (8.6,-2){\small$B_{G-Z}^{\ell_0}(F_{k+1}^\beta)$};
		\node[inner sep= 1pt,black](t15) at (10.1,-1.5){\small$F_{k+1}^\beta$};
		\node[inner sep= 1pt,darkpastelgreen](t16) at (11,-0.8){\small$P_{k+1}^\beta$};
		\node[inner sep= 1pt,black](t17) at (10.5,-3){\small$F_{k+2}^\beta$};
		\node[inner sep= 1pt,black](t18) at (12.5,-3){\small$\cdots$};
		\node[inner sep= 0.5pt,deepcarmine](t19) at (8.6,-2.7){\small Case $2$};
		\node[inner sep= 1pt,black](t20) at (-0.6,-2.6){\small$F_{s}^\alpha$};
		\node[inner sep= 1pt,black](t21) at (14.1,-2.6){\small$F_{s}^\beta$};
		\end{tikzpicture}
		\caption{An illustration of the proof of Lemma~\ref{lem-link-2pulp-adj}.}\label{fig-link-kraken}
	\end{figure}
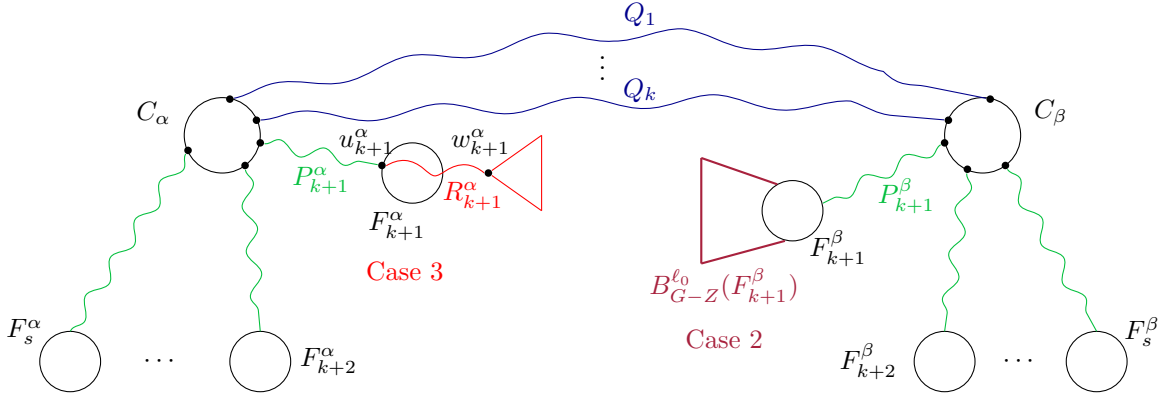
	
	\smallskip
	
	\noindent\emph{Case 1.} If $u_{k+1}^{\sigma}\in L$, then set $$X_{k+1}^{\sigma}=N(u_{k+1}^{\sigma})\cup V(P_{k+1}^{\sigma}).$$ 
	Note that in this case, $X_{k+1}^{\sigma}$ is a $(\Delta,10m+1)$-expansion of $v_{k+1}^{\sigma}$.
	
	\medskip
	
	Suppose then $u_{k+1}^{\sigma}\not\in L$ and so $F_{k+1}^\sigma\subseteq G-L$. Set 
	$$Z:=V(C_\alpha)\cup V(C_\beta)\cup \Big(\bigcup_{j=1}^sV(P_j^\alpha\cup P_j^\beta)\Big) \cup \Big(\bigcup_{i=1}^kV(Q_i)\Big).$$
	We wish to expand $F_{k+1}^{\sigma}$ in $G-Z$ while avoiding all unused legs $\cup_{i=k+2}^{s}P_{i}^\alpha\cup F_i^\alpha\cup P_{i}^\beta\cup F_i^\beta$.
	
	\smallskip
	
	\noindent\emph{Case 2.} If $F_{k+1}^{\sigma}$ does not intersect $L$ after expanding $\ell_0$ steps in $G-Z$, namely, $B_{G-Z}^{\ell_0}(V(F_{k+1}^{\sigma}))$ is disjoint from $L$. By the hypothesis, all low degree legs in $\mathsf{K}_\alpha, \mathsf{K}_\beta$ are pairwise a distance $(\log n)^{1/10}$ apart in $G-L$, and hence we see that $B_{G-Z}^{\ell_0}(V(F_{k+1}^{\sigma}))$ is disjoint from all unused low degree legs. Since
	$$|Z|\leq 2s+2s\cdot 10m+s\log^{t}n\leq 2\log^{t+1}n\le \frac{1}{4}\,\eps(|F_{k+1}^\sigma|)\cdot|F_{k+1}^\sigma|,$$
    we can apply Proposition~\ref{prop-exp-HL} with $(X,Y,W,r)_{\ref{prop-exp-HL}}=(F_{k+1}^\sigma,Z,\varnothing,\ell_0)$ to get that $$|B_{G-Z}^{\ell_0}(V(F_{k+1}^\sigma))|\geq \Delta.$$ 
    In this case, set 
    $$X_{k+1}^\sigma=B_{G-Z}^{\ell_0}(V(F_{k+1}^\sigma))\cup V(P_{k+1}^{\sigma}),$$ 
    which is a $(\Delta,11m+\ell_0)$-expansion of $v_{k+1}^{\sigma}$.
	
	\smallskip
	
	\noindent\emph{Case 3.} Suppose that $B_{G-Z}^{\ell_0}(V(F_{k+1}^{\sigma}))$ intersects $L$. Let $R_{k+1}^{\sigma}$ be a shortest $u_{k+1}^{\sigma}, L$-path in $B_{G-Z}^{\ell_0}(V(F_{k+1}^{\sigma}))$, which has length at most $m+\ell_0\le 2m$. In this case, let $w_{k+1}^{\sigma}$ be the endpoint of $R_{k+1}^{\sigma}$ in $L$ and set 
	$$X_{k+1}^\sigma=V(R_{k+1}^{\sigma})\cup N(w_{k+1}^{\sigma})\cup V(P_{k+1}^{\sigma}),$$
	which a $(\Delta,12m+1)$-expansion of $v_{k+1}^{\sigma}$.
	
	\smallskip
	
	Now, let 
	$$\hat{Z}=\Big(\bigcup_{i=1}^kV(Q_i)\Big)\cup \left(\bigcup_{i=k+2}^{s}V(P_{i}^\alpha\cup F_{i}^\alpha\cup P_i^\beta\cup F_i^\beta)\right).$$
	Note that
	$$|\hat{Z}|\le s\log^{t}n+2s\cdot (10m+(\log n)^{2t})\le (\log n)^{3t}.$$
	Trim $X_{k+1}^\sigma$, using Proposition~\ref{prop:trimming}, down to a $((\log n)^{4t},20m)$-expansion of $v_{k+1}^{\sigma}$ disjoint from $\hat{Z}$; call it $\hat{X}_{k+1}^\sigma$. Finally, apply Lemma~\ref{lem-finalconnect} with $$(D,U,F_1,F_2,v_1,v_2)_{\ref{lem-finalconnect}}=((\log n)^{4t}, \hat{Z},\hat{X}_{k+1}^\alpha,\hat{X}_{k+1}^\beta,v_{k+1}^\alpha,v_{k+1}^{\beta})$$
	to obtain the desired $v_{k+1}^{\alpha},v_{k+1}^{\beta}$-path $Q_{k+1}$ disjoint from $\bigcup_{i=k+2}^s V(P_{i}^\alpha\cup F_{i}^{\alpha}\cup P_{i}^\beta\cup F_{i}^{\beta})$. Carrying out the embeddings for all $k\in\{0,1,\dots,s-1\}$ finishes the proof.
\end{proof}


\section*{Acknowledgements}
We would like to thank Richard Montgomery for fruitful discussions.

\end{document}